\documentclass{amsart}
\usepackage{amsmath,amscd,xypic,amssymb,combelow,color,enumitem,graphicx,float,comment,mathtools,mathrsfs,tikz-cd,mathdots,amsbsy}

\newtheorem{theorem}[subsection]{Theorem}  

\newtheorem{proposition}[subsection]{Proposition}
\newtheorem{lemma}[subsection]{Lemma}

\newtheorem{definition}[subsection]{Definition}

\newtheorem{remark}[subsection]{Remark}

\def\fg{{\mathfrak{g}}}

\def\hg{{\widehat{\fg}}}

\def\BC{{\mathbb{C}}}

\def\BN{{\mathbb{N}}}

\def\BR{{\mathbb{R}}}

\def\BZ{{\mathbb{Z}}}

\def\CO{{\mathcal{O}}}

\def\nn{{\mathbb{N}^I}}

\def\nnn{{\mathbb{N}^{I \times \BZ}}}
\def\zzz{{\mathbb{Z}^{I \times \BZ}}}

\def\bpsi{{\boldsymbol{\psi}}}

\def\bu{{\boldsymbol{u}}}
\def\bv{{\boldsymbol{v}}}
\def\bw{{\boldsymbol{w}}}
\def\bx{{\boldsymbol{x}}}

\def\bU{{\boldsymbol{U}}}
\def\bV{{\boldsymbol{V}}}

\def\b0{{\boldsymbol{0}}}

\def\bV{{\mathbf{V}}}

\def\loccitt{\emph{loc.~cit.}}

\def\bx{\boldsymbol{x}}

\def\stab{\text{stab}}
\def\estab{\emph{stab}}

\def\bQ{\boldsymbol{Q}}
\def\bS{\boldsymbol{S}}

\begin{document}

\title[EXTREMAL MONOMIALS OF $q$-CHARACTERS]{\Large{\textbf{Extremal monomials of $q$-characters}}} 

\author[Andrei Negu\cb t]{Andrei Negu\cb t}

\address{École Polytechnique Fédérale de Lausanne (EPFL), Lausanne, Switzerland \newline \text{ } \ \ Simion Stoilow Institute of Mathematics (IMAR), Bucharest, Romania} 

\email{andrei.negut@gmail.com}

\begin{abstract}  In this short paper, we prove a conjecture of Frenkel-Hernandez, which states that $q$-characters of finite-dimensional simple modules of the quantum affine algebra $U_q(\widehat{\fg})$ are bounded by the Weyl group orbit of the leading monomial under Chari's braid group action. This can be regarded as a (partial) quantum affine generalization of the Weyl group invariance of characters of finite-dimensional representations of $\fg$.  \end{abstract}

\def\uppercasenonmath#1{}
\maketitle 

\section{Introduction}
\label{sec:intro}

\medskip 

\subsection{The setting}
\label{sub:setting}

Let $\fg$ be a simple complex Lie algebra, and let
\[
C = \left(c_{ij} = \frac {2d_{ij}}{d_{ii}} \in \BZ \right)_{i,j \in I}
\]
be its Cartan matrix, where $d_{ij} = (\alpha_i,\alpha_j)$ with respect to a henceforth fixed choice of simple roots $\{\alpha_i\}_{i \in I}$. As is customary, we normalize the $d_{ij}$ such that the $d_{ii}$ are even positive integers with gcd $2$. We will be interested in type 1 finite-dimensional simple modules of the untwisted quantum affine algebra (we work with $q \in \BC^* \backslash \sqrt[\BN]{1}$)
\begin{equation}
\label{eqn:simple module}
U_q(\widehat{\fg}) \curvearrowright L(\bpsi)
\end{equation}
As shown in \cite{CP}, the simple modules above are indexed by a monomial $\bpsi$ in symbols
\begin{equation}
\label{eqn:fundamental l-weights}
\{ Y_{k,c} \}_{k \in I, c \in \BC^*}
\end{equation}
One calls the $L(Y_{k,c})$ fundamental representations, since any module \eqref{eqn:simple module} is a subquotient of a tensor product of fundamentals. We will also consider the expressions
\begin{equation}
\label{eqn:simple l-weights}
A_{i,x}^{-1} = Y_{i,xq_i}^{-1} Y_{i,xq_i^{-1}}^{-1} \prod_{j \neq i} \prod_{\ell = \frac {c_{ji}+1}2}^{-\frac {c_{ji}+1}2} Y_{j,xq^{\ell d_{jj}}}
\end{equation}
with $q_i = q^{d_i}$, $d_i = \frac {d_{ii}}2$. As shown in \cite{FM}, the $q$-character (\cite{FR}) of \eqref{eqn:simple module} is an expression
\begin{equation}
\label{eqn:q-character}
\chi_q(L(\bpsi)) = \bpsi \sum_{\text{multisets } \bx = \{x_{i1},x_{i2},\dots\}_{\forall i \in I} \subset \BC^*} \mu_{\bx}^{\bpsi} \prod_{i \in I} A_{i,x_{i1}}^{-1} A_{i,x_{i2}}^{-1} \dots
\end{equation}
The non-negative integers $\mu_{\bx}^{\bpsi}$ are of great interest, and they were interpreted geometrically in the following seminal result of \cite{HL}, which applies the quiver Grassmannian formulas of \cite{DWZ} to a particular quiver built out of $\fg$ (see Subsection~\ref{sub:quiver variety old}).

\medskip 

\begin{theorem}
\label{thm:hl}

(\cite{HL} and Proposition~\ref{prop:nilpotent}) For any $\bpsi = Y_{k,c}$ with $k \in I$ and $c \in \BC^*$,
\begin{equation}
\label{eqn:hl}
\mu_{\bx}^{\bpsi} = \chi(N_{\bx,\bpsi}^{\estab})
\end{equation}
where $N_{\bx,\bpsi}^{\estab}$ is a certain variety that we will recall in Subsection \ref{sub:quiver variety old}. 

\end{theorem}

\medskip 

\noindent For more general (so-called reachable) $\bpsi$, Theorem \ref{thm:hl} was proved by combining the results of \cite{DWZ} and \cite{KKOP}. We also refer to \cite{N1, N2} for some recent algebraic/geometric realizations of the numbers $\mu_{\bx}^{\bpsi}$ that go beyond finite type $\fg$ and reachable $\bpsi$.

\medskip 

\subsection{The extremal monomial conjecture}
\label{sub:conjecture}

We will consider the Weyl group $W$ associated to $\fg$, generated by simple reflections $\{s_i\}_{i \in I}$ in the hyperplanes $\alpha^\perp_i$, and the corresponding braid group. It was shown in \cite{C} that the assignment
\begin{equation}
\label{eqn:weyl action}
S_i(Y_{j,x}) = Y_{j,x} A_{i,x q_i^{-1}}^{-\delta_{ij}} \quad \Rightarrow \quad S_i(A_{j,x}^{-1}) = \begin{cases} A_{i,xq_i^{-2}}& \text{if }i = j \\ A_{j,x}^{-1} \prod_{\ell =\frac {c_{ij}}2}^{-\frac {c_{ij}}2-1} A^{-1}_{i,xq^{\ell d_{ii}}}&\text{if }i \neq j \end{cases}
\end{equation}
induces a braid group action on the set of Laurent monomials in the symbols \eqref{eqn:fundamental l-weights} (our $S_i$ are the $T_i^{-1}$ of \cite{FH}). Therefore, we may define $S_w$ for any Weyl group element $w \in W$, simply by taking a reduced decomposition of $w$ into simple reflections. With this in mind, our main result (proposed in \cite[Conjecture 4.4]{FH}) is the following.

\medskip 

\begin{theorem}
\label{thm:main}

For any multiset $\bx = \{x_{i1},x_{i2},...\}_{i \in I} \subset \BC^*$, we have $\mu_{\bx}^{\bpsi} = 0$ unless
\begin{equation}
\label{eqn:main}
S_w \left(\bpsi \prod_{i \in I} A_{i,x_{i1}}^{-1} A_{i,x_{i2}}^{-1} \dots \right) \in \bpsi \prod_{i \in I} \mathop{\prod_{\text{various } c \in \BC^*}}_{\text{with multiplicities}} A_{i,c}^{-1}
\end{equation}
for all $w \in W$ (we note that \eqref{eqn:main} was already proved for $w\in \{s_i,w_0\}_{i \in I}$ in \cite{FH}).
	
\end{theorem}

\medskip 

\noindent In other words, the only monomials that appear in $\chi_q(L(\bpsi))$ are those which lie in the intersection of $|W|$ many cones emanating from $\{S_w^{-1}(\bpsi)\}_{w \in W}$. This provides a necessary condition for which monomials can appear in the $q$-character of simple $U_q(\hg)$-modules. This condition is not sufficient in general, nor are the multiplicities of various monomials generally invariant under the $S_w$'s (however, see \cite{C, CM} for the multiplicity $1$ property of extremal monomials).

\medskip

\noindent As explained in \cite{FH}, Theorem \ref{thm:main} has important consequences for finite-dimensional $U_q(\hg)$-modules, such as the polynomiality of eigenvalues of the $X$-series (Conjecture 6.9 of \loccitt), which is important in the study of generalized Baxter operators.

\medskip 

\noindent The idea of the proof of Theorem \ref{thm:main} is to enlarge $N_{\bx,\bpsi}^{\stab}$ using certain generalizations of graded Nakajima quiver varieties (\cite{Na2}), which we denote by 
\begin{equation}
\label{eqn:p intro}
P_{\bv,\bw}^{\theta\text{-stab}}
\end{equation} 
and to define functions $\pmb{\mathscr{S}}_i$ as in \eqref{eqn:technical} between them that underlie the assignments \eqref{eqn:weyl action}. Thus, if we had $\mu_{\bx}^{\bpsi} \neq 0$ and some $S_w = S_{i_t} \dots S_{i_1}$ that failed property \eqref{eqn:main}, this would correspond to a composition $\pmb{\mathscr{S}}_{i_t} \circ \dots \circ \pmb{\mathscr{S}}_{i_1}$ of functions which send a non-empty quiver variety to an empty quiver variety (contradiction). We note that we will work with \eqref{eqn:p intro} as sets of closed points, and thus abuse the term ``variety". It is reasonable to expect that one can make \eqref{eqn:p intro} into quasiprojective varieties using geometric invariant theory, but this goes beyond the scope of this short paper.

\medskip 

\subsection{Acknowledgements} I would like to thank David Hernandez for his great help in understanding the representation theory of quantum affine algebras.

\medskip 

\section{The proof}
\label{sec:proof}

The set $\BN$ is assumed to contain 0 throughout the paper. As explained in \cite[Remark 4.5]{FH}, it suffices to prove Theorem \ref{thm:main} in the particular case 
\begin{equation}
\label{eqn:restriction}
\bpsi=Y_{k,c}
\end{equation} 
for any $k \in I$ and $c \in \BC^*$. For this $\bpsi$, we can only have $\mu_{\bx}^{\bpsi} \neq 0$ for those multisets $\bx$ which lie in $cq^{\BZ}$. Therefore, we will henceforth replace multisets $\bx$ by the tuples 
\begin{equation}
\label{eqn:replace}
\bx \leadsto \bv = (v_i^a)_{i \in I}^{a \in \BZ} \in \nnn
\end{equation}
\[
v_i^a = \Big(\text{number of times }cq^a \text{ appears among } x_{i1},x_{i2},\dots \Big), \quad \forall i \in I, a \in \BZ
\]
(throughout the present paper, the notation $\nnn$ and $\zzz$ will refer to $\infty$-tuples of numbers with finite support, i.e. only finitely many entries can be non-zero).

\medskip 

\subsection{Quiver Grassmannians}
\label{sub:quiver variety old}

Let us describe the varieties that feature in Theorem \ref{thm:hl}, although we will do so in the equivalent language of \cite{N2} rather than the original one of \cite{HL}. Consider the quiver $\bQ$ with vertex set $I \times \BZ$ and arrows $(j,a-d_{ij}) \leftarrow (i,a)$ for all $i,j \in I$ such that $c_{ij} \neq 0$, and $a \in \BZ$. A representation of the quiver $\bQ$ consists of a collection of vector spaces $\{V_i^a\}_{i \in I}^{a \in \BZ}$ and linear maps
\begin{equation}
\label{eqn:quiver representation}
\bV = \left\{V_{j}^{a-d_{ij}} \xleftarrow{^{a-d_{ij}}_j\square_i^a} V_i^a \right\}_{i,j \in I, c_{ij} \neq 0}^{a \in \BZ}
\end{equation}
We will abbreviate for all $i,i',i'' \in I$, $a \in \BZ$ and $\ell \geq 0$
\begin{equation}
\label{eqn:abbreviate 1}
^{a-\ell d_{ii}}_i\square_i^a = \left( ^{a-\ell d_{ii}}_i\square_i^{a-(\ell-1)d_{ii}} \right) \circ \dots \circ \left( ^{a-2d_{ii}}_i\square_i^{a-d_{ii}} \right) \circ \left( ^{a-d_{ii}}_i\square_i^a \right)
\end{equation}
\begin{equation}
\label{eqn:abbreviate 2}
\cdots ^{a}_{i}\square_{i'}^{a+d_{ii'}} \square_{i''}^{a+d_{ii'}+d_{i'i''}} \cdots =  \dots \circ \left( ^{a}_i\square_{i'}^{a+d_{ii'}} \right) \circ \left( ^{a+d_{ii'}}_{i'}\square_{i''}^{a+d_{ii'}+d_{i'i''}} \right) \circ \dots
\end{equation}
The dimension of a quiver representation \eqref{eqn:quiver representation} is $\dim \bV = (\dim V_i^a)_{i \in I}^{a \in \BZ} \in \nnn$. The next step is to consider framed quiver representations; due to our specific choice of \eqref{eqn:restriction}, we will enhance the datum \eqref{eqn:quiver representation} by adding to it a ``framing" vector $\xi \in V_k^{d_k}$. Let $M_{\bx,\bpsi}$ be the stack of $\bv$-dimensional framed quiver representations (with $\bx$ and $\bv$ related by the one-to-one correspondence in \eqref{eqn:replace}), and let
\begin{equation}
\label{eqn:substack}
N_{\bx,\bpsi} \hookrightarrow M_{\bx,\bpsi}
\end{equation}
denote the closed substack cut out by the following equation for all $i \in I$ and $a \in \BZ$
\begin{equation}
\label{eqn:1}
\sum_{j \neq i} \sum_{\ell = 0}^{-c_{ij}-1} 
{^{a+d_{ii}}_i\square^{a-\ell d_{ii}-2d_{ij}}_i} \square^{a-\ell d_{ii}-d_{ij}}_j \square^{a-\ell d_{ii}}_i \square_{i}^{a}  = 0
\end{equation}
(see notation \eqref{eqn:abbreviate 1}-\eqref{eqn:abbreviate 2}), together with the following equation whenever $c_{ij}<0$
\begin{equation}
\label{eqn:2}
{^{a+d_{ij}}_j\square_j^{a-d_{ij}}} \square_i^a + {^{a+d_{ij}}_j\square_i^{a+2d_{ij}}} \square_i^a  = 0
\end{equation} 
for all $a \in \BZ$, and finally with (recall that $d_k = \frac {d_{kk}}2$)
\begin{equation}
\label{eqn:3}
\left( ^{-d_k}_k\square_k^{d_k} \right) ( \xi ) = 0
\end{equation}
A framed quiver representation will be called stable if it has no proper subrepresentations which contain $\xi$. Intersecting \eqref{eqn:substack} with the open locus of stable framed quiver representations gives us a closed embedding
\begin{equation}
	\label{eqn:subvariety}
	N_{\bx,\bpsi}^{\stab} \hookrightarrow M_{\bx,\bpsi}^{\stab}
\end{equation}
Meanwhile, \cite[Theorem 4.8]{HL} proves the following version of \eqref{eqn:hl}
\begin{equation}
	\label{eqn:hl equivalent}
\mu_{\bx}^{\bpsi} = \chi(\tilde{N}_{\bx,\bpsi}^{\stab})
\end{equation}
where the variety in the right-hand side is determined in the following result.

\medskip 

\begin{proposition}
\label{prop:nilpotent}

Let $\tilde{N}^{\estab}_{\bx,\bpsi}$ be defined as $N_{\bx,\bpsi}^{\estab}$, with the following modifications:

\medskip 

\begin{itemize}[leftmargin=*]
	
\item requiring points $\bV$ as in \eqref{eqn:quiver representation} to be \emph{nilpotent}, i.e. the composition of linear maps corresponding to any path in $\bQ$ of sufficiently high length acts by $0$ on $\bV$;

\medskip 

\item replacing relation \eqref{eqn:3} by 
\begin{equation}
\label{eqn:3 equiv}
\sum_{\alpha  \text{ path from }(k,d_k) \text{ to }(k,-d_k)} \gamma_{\alpha} \square_{\alpha}(\xi) = 0
\end{equation}
where $\gamma_{\alpha} \in \BC$ are generic (the infinite sum is well-defined due to the first bullet). 

\end{itemize} 

\medskip 

\noindent Then we have $\tilde{N}_{\bx,\bpsi}^{\estab} = N_{\bx,\bpsi}^{\estab}$.

\end{proposition}

\medskip 

\begin{proof} We first show that any $\bV$ in either $N_{\bx,\bpsi}^{\stab}$ or $\tilde{N}_{\bx,\bpsi}^{\stab}$ is nilpotent, meaning that the first bullet in the statement of the Proposition is actually redundant. To this end, note that an arbitrary composition of the linear maps $_j^t\square_i^s$ takes the form
\begin{equation}
\label{eqn:path}
_{i_n}^{a+M-\ell_0 d_{i_0i_0}}\square_{i_n}^{a+M-\ell_0 d_{i_0i_0}+\ell_n d_{i_ni_n}} \cdots 
\end{equation}
\[
\cdots _{i_2}^{a-\ell_0 d_{i_0i_0} - d_{i_0i_1}-\ell_1 d_{i_1i_1} - d_{i_1i_2}}\square^{a-\ell_0 d_{i_0i_0} - d_{i_0i_1}-\ell_1 d_{i_1i_1}}_{i_1}\square^{a-\ell_0 d_{i_0i_0} - d_{i_0i_1}}_{i_1}\square_{i_0}^{a-\ell_0 d_{i_0i_0}}\square^a_{i_0}
\]
for various $i_0, i_1, \dots, i_n$ in $I$ which are successively connected (i.e. $c_{i_{s-1}i_{s}} < 0$ for all $s \in \{1,\dots,n\}$), and various  $\ell_0,\dots,\ell_n \geq 0$. Above, we write
\begin{equation}
\label{eqn:m}
M = - d_{i_0i_1} - \ell_1 d_{i_1i_1} - \dots - d_{i_{n-1}i_n} - \ell_n d_{i_ni_n} 
\end{equation} 
By using \eqref{eqn:2}, we may move blocks of loops to the right of the composition \eqref{eqn:path}, which allows us to restrict attention only to those compositions which satisfy
\begin{equation}
\label{eqn:inequalities}
\ell_s \leq - c_{i_si_{s-1}} -1, \quad \forall s \in \{1,\dots,n
\}
\end{equation}
Therefore, we have the inequality
\[
M \geq d_{i_0i_1} + d_{i_1i_1} + \dots + d_{i_{n-1}i_n} +  d_{i_ni_n} = d_{i_n} - d_{i_0} + \sum_{s=1}^{n} \left(d_{i_{s-1}i_{s}} + d_{i_{s-1}} + d_{i_{s}} \right) \geq n-2
\]
where the last inequality uses the fact that $d_{ij} + d_i + d_j \geq 1$ whenever $c_{ij} < 0$ in a Dynkin diagram. This means that any composition of the linear maps $_j^t\square_i^s$ corresponds equals $\pm$ a certain number of loops at $i_0$ composed with a succession of maps which increases the upper index $a$ by at least $n-2$. If we choose $n$ so that
\[
n-2 > \text{max}\left\{ a-b \Big| V_i^a \neq 0, V_j^b \neq 0  \text{ for various }i,j \text{ in \eqref{eqn:quiver representation}}\right\}
\]
then the aforementioned succession of linear maps must be identically $0$. Similarly, a large enough number of successive loops at $i_0$ acts by $0$ on $\bV$, for the same reason.

\medskip 

\noindent We must now prove the equivalence of conditions \eqref{eqn:3} and \eqref{eqn:3 equiv}. To this end, we must prove that (the composition of linear maps corresponding to) any path in $\bQ$ from $(k,d_k)$ to $(k,-d_k)$ is equivalent, modulo relations \eqref{eqn:2}, to a path that begins with the arrow $(k,-d_k) \leftarrow (k,d_k)$. As we have already seen, any such path can be assumed to be of the form \eqref{eqn:path} with $i_0 = i_n = k$ and
\begin{equation}
\label{eqn:conditions}
a = d_k, \quad a+M - \ell_0 d_{i_0i_0} = -d_k
\end{equation}
where the numbers $\ell_1,\dots,\ell_n$ satisfy \eqref{eqn:inequalities} and $M$ is defined in \eqref{eqn:m}. Since we have $i_0=i_n = k$, we can strengthen the inequality satisfied by $M$ to the inequality $M \geq n$. Because of this, \eqref{eqn:conditions} implies $\ell_0 > 0$, precisely as we needed to prove. \end{proof}

\medskip 

\subsection{Quiver varieties}
\label{sub:quiver variety new}

We will now generalize the varieties from the previous Subsection by including general framing maps
\begin{equation}
\label{eqn:coframing}
V_i^{a+d_i} \xleftarrow{_i^{a+d_i}A} W_i^a \xleftarrow{B_i^{a-d_i}} V_i^{a-d_i}
\end{equation}
where $W_i^a$ are vector spaces whose dimensions are indexed by $\bw = (w_i^a)_{i \in I}^{a\in \BZ} \in \nnn$. 

\medskip 

\begin{definition}
\label{def:new}

Let $P_{\bv,\bw}$ denote the stack parameterizing framed $\bv$-dimensional quiver representations, i.e. collections of data \eqref{eqn:quiver representation} and \eqref{eqn:coframing} that satisfy relations
\begin{equation}
\label{eqn:1 bis}
\left( \sum_{j \neq i} \sum_{\ell = 0}^{-c_{ij}-1} 
{^{a+d_{ii}}_i\square^{a-\ell d_{ii}-2d_{ij}}_i} \square^{a-\ell d_{ii}-d_{ij}}_j \square^{a-\ell d_{ii}}_i \square_{i}^{a} \right) + \ {_i^{a+d_{ii}}A} B_i^a = 0
\end{equation}
\begin{equation}
	\label{eqn:4}
^{a-d_{i}}_i\square_i^{a+d_i} A = 0
\end{equation}
\begin{equation}
	\label{eqn:5}
B^{a-d_i}_i \square_i^{a+d_i} = 0
\end{equation}
for all $i \in I, a \in \BZ$, together with relation \eqref{eqn:2}. Above and henceforth, we write
\begin{align*} 
&{_i^{a+d_{ii}}A} B_i^a = \Big( {_i^{a+d_{ii}}A} \Big) \circ \Big( B_i^a \Big) \\
&^{a-d_{i}}_i\square_i^{a+d_i} A = \Big( {^{a-d_{i}}_i\square_i^{a+d_i}} \Big) \circ \Big({_i^{a+d_i}A}\Big) \\
&B^{a-d_i}_i \square_i^{a+d_i} = \Big( {B^{a-d_i}_i} \Big) \circ \Big( {^{a-d_i}_i\square_i^{a+d_i}} \Big) 
\end{align*}
i.e. any adjacent identical indices are contracted (akin to \eqref{eqn:abbreviate 2}).

\end{definition}
 
\medskip 

\noindent We will let $P_{\bv,\bw}^{\stab} \subset P_{\bv,\bw}$ denote the open subset of stable points, i.e. those with no proper subrepresentations that include the images of all the $A$ maps. 

\medskip 

\begin{proposition}
\label{prop:combi}

The assignment $(_j\square_i, \xi) \leadsto (_j\square_i, {_iA}(z) = \delta_{ik} z\xi, B_i = 0)$ yields 
\begin{equation}
\label{eqn:closed embedding} 
N_{\bx,\bpsi}^{\estab} \cong P_{\bv,\bw}^{\estab}
\end{equation}
where $\bx$ and $\bv$ are related as in \eqref{eqn:replace}, and for the specific $\bpsi$ of \eqref{eqn:restriction} we let $\bw \in \nnn$ denote the tuple with 1 on position $(k,0)$, and 0 everywhere else. 

\end{proposition}

\medskip 

\begin{proof} The Proposition is equivalent to the fact that any point of $P_{\bv,\bw}^{\stab}$ satisfies $B = 0$ (this only holds for our $\bw$). By stability, this is equivalent to the fact that
\begin{equation}
\label{eqn:vanishing}
B {_k^{-d_k}\square_k^{d_k}} A = 0
\end{equation}
where $\square$ is any path in $\bQ$ from $(k,d_k)$ to $(k,-d_k)$. As we showed in the final paragraph of the proof of Proposition~\ref{prop:nilpotent}, the path in question can be assumed to begin with the arrow $(k,d_k) \mapsto (k,-d_k)$. Therefore, equation \eqref{eqn:4} implies \eqref{eqn:vanishing}. \end{proof}
 
\medskip 

\begin{remark}
\label{rem:nakajima}

The spaces $P_{\bv,\bw}^{\estab}$ are to Nakajima quiver varieties as the algebras of \cite{GLS, HL} are to preprojective algebras of quivers (see also \cite{ST}). When $\fg$ is simply-laced, $P_{\bv,\bw}^{\estab}$ coincide with the graded Nakajima quiver varieties of \cite{Na2}. For general $\fg$, various flavors of this construction appeared in \cite{TY, Y}, where the authors also explored reflections akin to the ones that we will shortly construct, and in \cite{VV} in the context of critical $K$-theory (which is highly relevant to the study of $q$-characters).

\end{remark}

\medskip

\subsection{$\theta$-stability}
\label{sub:stability}

We will need to consider more general stability conditions, following the work of King (\cite{K}) and Nakajima (\cite{Na0}). Let $\{\omega_i\}_{i \in I}$ denote the fundamental weights corresponding to our fixed choice of simple roots, and consider 
\begin{equation}
\label{eqn:stability condition}
\theta = \sum_{i \in I} \theta_i \omega_i 
\end{equation}
with $\theta_i \in \BR$. We will only consider generic $\theta$'s, i.e. those which are not situated on any root hyperplane. For any $\bv = (v_{i}^a)_{i \in I}^{a \in \BZ} \in \BN^{I \times \BZ}$, we let $v_i = \sum_{a \in \BZ} v_i^a$ and define
\begin{equation}
\label{eqn:stability pairing}
(\theta , \bv) = \sum_{i \in I} d_i \theta_i v_i
\end{equation}

 \medskip 

\begin{definition}
\label{def:stable}

A point of $P_{\bv,\bw}$ is called $\theta$-stable whenever the following hold:

\medskip  

\begin{itemize}[leftmargin=*]	
	
\item if $\{U_i^a \subseteq V_i^a\}_{i \in I}^{a \in \BZ}$ is a $\bu$-dimensional subrepresentation contained in $\emph{Ker }B$, then 
\begin{equation}
\label{eqn:stable 1}
(\theta, \bu) \leq 0
\end{equation}

\medskip   

\item if $\{U_i^a \subseteq V_i^a\}_{i \in I}^{a \in \BZ}$ is a $\bu$-dimensional subrepresentation that contains $\emph{Im }A$, then 
\begin{equation}
\label{eqn:stable 2}
(\theta, \bv-\bu) \geq 0
\end{equation}

\end{itemize} 

\medskip 
	
\noindent (the inequalities above should have been written as ``$<$ if $\bu \neq \b0$" and ``$>$ if $\bu \neq \bv$", respectively, but the genericity of $\theta$ makes this unnecessary). When $\theta$ satisfies $\theta_i < 0$ for all $i \in I$, $\theta$-stability is equivalent to the stability of Subsection \ref{sub:quiver variety new}.
	
\end{definition}

\medskip  

\noindent Let $P_{\bv,\bw}^{\theta\text{-stab}}$ be the subset of $P_{\bv,\bw}$ consisting of $\theta$-stable points. When $\fg$ is simply-laced (i.e. $d_{ii}=2$ for all $i$) it is easy to prove that any $\theta$-stable point must satisfy
\begin{equation}
\label{eqn:vanish}
^{a-d_{ii}}_i\square_i^a = 0, \quad \forall i\in I, \ a \in \BZ
\end{equation}
and thus the $P_{\bv,\bw}^{\theta\text{-stab}}$ are none other than Nakajima's graded quiver varieties (\cite{Na0}). 

\medskip 

\subsection{Reflections}
\label{sub:reflections}

The braid group action of \eqref{eqn:weyl action} induces a braid group action
\begin{equation}
\label{eqn:weyl dimension vector def}
\bS_i : \zzz \rightarrow \zzz
\end{equation}
where $\bS_{i} (\{v_i^a\}_{i \in I}^{a\in \BZ}) = \{\bar{v}_i^a\}_{i \in I}^{a\in \BZ}$ is determined by Chari's maps \eqref{eqn:weyl action} as follows:
\begin{equation}
\label{eqn:weyl dimension vector}
S_i\left(\bpsi \prod_{j \in I} \prod_{a \in \BZ} A_{j,cq^a}^{-v_j^a} \right) = \bpsi \prod_{j \in I} \prod_{a \in \BZ} A_{j,cq^a}^{-\bar{v}_j^a}
\end{equation}
(the highest weight is determined by $\bw$ via the formula $\bpsi = \prod^{a \in \BZ}_{i \in I} Y_{i,a}^{w_i^a}$, and will be suppressed from the notation of $\bS_i$; for the applications in this paper, we may restrict to $\bpsi$ as in \eqref{eqn:restriction}). Moreover, we have a Weyl group action $s_i(\theta) = \bar{\theta}$, where
\[
\bar{\theta}_j = \theta_j - \theta_i c_{ji}, \quad \forall j\in I
\]
Our main technical result is the following construction, which is inspired by Nakajima's reflection functors (which were also studied and developed by many other authors, see \cite{L, M, Na1, TY, Y}).

\medskip  

\begin{proposition}
\label{prop:technical}
	
If $\theta_i < 0$, then for any $\bv,\bw \in \nnn$ we have a function \footnote{If $\theta_i > 0$, then by reversing the arrows in Subsection \ref{sub:technical}, one would obtain a function 
\begin{equation}
\label{eqn:technical bis}
\pmb{\mathscr{S}}_i : P_{\bv,\bw}^{\theta\text{-stab}} \rightarrow P_{\bS_i^{-1}(\bv),\bw}^{s_i(\theta)\text{-stab}}
\end{equation}
We expect that \eqref{eqn:technical bis} and \eqref{eqn:technical} satisfy the braid relations (see also \cite{TY, Y}). In fact, for simply-laced $\fg$, this follows from \cite{L,M,Na1} by replacing Nakajima quiver varieties with their graded versions.}
\begin{equation}
\label{eqn:technical}
\pmb{\mathscr{S}}_i : P_{\bv,\bw}^{\theta\emph{-stab}} \rightarrow P_{\bS_i(\bv),\bw}^{s_i(\theta)\emph{-stab}}
\end{equation}
	
\end{proposition} 

\medskip  

\noindent The existence of the function \eqref{eqn:technical} entails the fact that if the domain is nonempty, then the codomain must be nonempty as well. In particular, this requires one to have $\bS_i(\bv) \in \nnn$. We will prove Proposition \ref{prop:technical} in the following Subsection, but let us first show how to deduce our main Theorem from it.

\medskip  

\begin{proof} \emph{of Theorem \ref{thm:main}}: Let $\bw \in \nnn$ be the tuple with 1 on position $(k,0)$ and 0 everywhere else, and let $\bpsi$ be given by \eqref{eqn:restriction}. Consider any $w \in W$ and let us consider a reduced decomposition $w = s_{i_t} \dots s_{i_1}$. It is well-known that the roots
\[
\alpha_{i_1}, s_{i_1}(\alpha_{i_2}),\dots,s_{i_1}\dots s_{i_{t-1}}(\alpha_{i_t})
\]
are all positive. Therefore, if we fix any weight $\theta$ in the (interior of the) negative Weyl chamber, we will have
\[
(\theta, s_{i_1}\dots s_{i_{u-1}}(\alpha_{i_u})) < 0
\]
for all $u \in \{1,\dots,t\}$. In turn, this implies that the weight
\[
s_{i_{u-1}} \dots s_{i_1} (\theta)
\]
has negative coefficient of $\omega_{i_u}$, for all $u \in \{1,\dots,t\}$. Thus, Proposition \ref{prop:technical} implies that we have a chain of functions
\begin{equation}
\label{eqn:chain}
P_{\bv,\bw}^{\theta\text{-stab}} \xrightarrow{\pmb{\mathscr{S}}_{i_1}} P_{\bS_{i_1}(\bv),\bw}^{s_{i_1}(\theta) \text{-stab}} \xrightarrow{\pmb{\mathscr{S}}_{i_2}} P_{\bS_{i_2} \bS_{i_1}(\bv),\bw}^{s_{i_2}s_{i_1}(\theta)\text{-stab}} \xrightarrow{\pmb{\mathscr{S}}_{i_3}} \dots \xrightarrow{\pmb{\mathscr{S}}_{i_t}} P_{\bS_{i_t} \dots \bS_{i_1}(\bv),\bw}^{s_{i_t} \dots s_{i_1}(\theta)\text{-stab}}
\end{equation}
Thus, LHS $\neq \varnothing \Rightarrow$ RHS $\neq \varnothing$. Let us now go back to the setting of Theorem \ref{thm:main}. Assume $\mu_{\bx}^{\bpsi} \neq 0$ for a certain multiset $\bx$ of complex numbers, which corresponds to a dimension vector $\bv \in \nnn$ as in \eqref{eqn:replace}. Then Theorem \ref{thm:hl} implies that $N_{\bx,\bpsi}^{\stab} \neq \varnothing$, which by \eqref{eqn:closed embedding} implies that $P_{\bv,\bw}^{\stab} = P_{\bv,\bw}^{\theta\text{-stab}} \neq \varnothing$. The existence of \eqref{eqn:chain} implies that 
\[
P_{\bS_w(\bv),\bw}^{w(\theta) \text{-stab}} \neq \varnothing \quad \Rightarrow \quad \bS_w(\bv) \in \nnn
\]
Since $\bS_w(\bv)$ is determined by \eqref{eqn:weyl dimension vector}, the conclusion above precisely implies \eqref{eqn:main}. \end{proof}
 
\medskip  
 
\subsection{Constructing $\pmb{\mathscr{S}}_i$:}
\label{sub:technical}

We generalize the idea of \cite[Section 3(ii)]{Na1}. For any
\begin{equation}
\label{eqn:point p}
\bV = \left\{ V_{j}^{a-d_{ij}} \xleftarrow{^{a-d_{ij}}_j\square_i^a} V_i^a, \ V_i^{a+d_i} \xleftarrow{_i^{a+d_i}A} W_i^a \xleftarrow{B_i^{a-d_i}} V_i^{a-d_i} \right\}_{i,j \in I, c_{ij} \neq 0}^{a \in \BZ} 
\end{equation}
viewed as a point of $P_{\bv,\bw}^{\theta\text{-stab}}$, consider the linear map
\begin{equation}
\label{eqn:complex}
\Phi_i^a : W_i^{a+d_{i}} \bigoplus_{j \neq i} \bigoplus_{\ell = \frac {c_{ij}}2+1}^{-\frac {c_{ij}}2} V_j^{a+\ell d_{ii}} \xrightarrow{\begin{bmatrix} _i^{a+d_{ii}}A \\ {^{a+d_{ii}}_i\square_i^{a-d_{ij}+\ell d_{ii}}} \square_j^{a+\ell d_{ii}} \end{bmatrix}^T} V_i^{a+d_{ii}}
\end{equation}

\medskip  

\begin{lemma}
\label{lem:surj}

The map $\Phi_i^a$ is surjective if $\theta_i < 0$.

\end{lemma}

\medskip   

\begin{proof} Consider the collection of subspaces $\bU = (\text{Im }\Phi_i^{a-d_{ii}}, V_j^a)_{j \neq i}^{a \in \BZ}$, and we claim that it is a subrepresentation of $\bV$. Once we show this, since $\bU$ also contains the image of all the $A$ maps, \eqref{eqn:stable 2} implies that $\bU = \bV$, as needed. Clearly, $\bU$ already contains the images of all the maps $^{a-d_{ij}}_j\square_i^a$ for $i \neq j$, so it remains to show that
\begin{align*}
&^{a+d_{ii}}_i\square_j^{a+d_{ij}+d_{ii}} \left(V_j^{a+d_{ij}+d_{ii}} \right) \subseteq \text{Im }\Phi_i^{a} \\
&^{a}_i\square_i^{a+d_{ii}} \left( \text{Im }\Phi_i^a \right) \subseteq \text{Im }\Phi_i^{a-d_{ii}}
\end{align*}
One proves the first inclusion by considering $\Phi_i^a(0,V_j^{a+d_{ij}+d_{ii}},0,\dots,0) \subseteq V_i^{a+d_{ii}}$. The second inclusion is an immediate consequence of the commutativity of 
\begin{equation}
\label{eqn:commutative diagram 1}
\xymatrix{ W_i^{a+d_{i}} \bigoplus_{j \neq i} \bigoplus_{\ell = \frac {c_{ij}}2+1}^{-\frac {c_{ij}}2} V_j^{a+\ell d_{ii}} \ar[d]_{\Upsilon_i} \ar[r]^-{\Phi_i^a} & V_i^{a+d_{ii}} \ar[d]^{^{a}_i\square_i^{a+d_{ii}} }\\ W_i^{a-d_i} \bigoplus_{j \neq i} \bigoplus_{\ell = \frac {c_{ij}}2}^{-\frac {c_{ij}}2-1} V_j^{a+\ell d_{ii}} \ar[r]^-{\Phi_i^{a-d_{ii}}} &V_i^{a} }
\end{equation}
where $\Upsilon_i$ is defined as the linear map which sends $ W_i^{a+d_{i}}$ to $0$, sends each $V_j^{a+\ell d_{ii}}$ identically onto itself if $\ell < -\frac {c_{ij}}2$, and finally sends 
\begin{equation}
\label{eqn:iteration}
V_j^{a-d_{ij}} \xrightarrow{- \left(^{a+d_{ij}}_j\square_j^{a-d_{ij}}\right)} V_j^{a+d_{ij}}
\end{equation}
The commutativity of \eqref{eqn:commutative diagram 1} uses relations \eqref{eqn:2} and \eqref{eqn:4}. If we wanted to stack several diagrams \eqref{eqn:commutative diagram 1} on top of each other, the induced iteration of the map $\Upsilon_i$ would take each $V_j^{\bullet}$ in the domain to the unique $V_j^{\bullet'}$ in the codomain such that $\bullet'-\bullet = 2d_{ij}k$ for some $k \in \BN$, via the $k$-fold iteration of the map \eqref{eqn:iteration}. \end{proof}

\medskip 

\noindent As a consequence of \eqref{eqn:1 bis}, it is easy to see that the linear map
\begin{equation}
\label{eqn:upsilon}
\Psi_i^a : V_i^a \xrightarrow{\begin{bmatrix} B_i^a \\ {^{a+\ell d_{ii}}_j\square_i^{a+d_{ij}+\ell d_{ii}}}\square_i^{a}  \end{bmatrix}} W_i^{a+d_{i}} \bigoplus_{j \neq i} \bigoplus_{\ell = \frac {c_{ij}}2+1}^{-\frac {c_{ij}}2} V_j^{a+\ell d_{ii}}
\end{equation}
satisfies $\Phi_i^a \circ \Psi_i^a = 0$, and thus we obtain an induced map
\begin{equation}
\label{eqn:induced map}
V_i^a \rightarrow \bar{V}_i^a = \text{Ker }\Phi_i^a
\end{equation}

\medskip 

\begin{lemma}
\label{lem:rep}

For any $\bV$ as in \eqref{eqn:point p}, let $\bar{V}_i^a = \emph{Ker }\Phi_i^a$ and consider the maps
\begin{align} 
&_j^{a-d_{ij}}\bar{\square}_i^a : \bar{V}_i^a \xrightarrow{\text{projection onto summand}} V_j^{a-d_{ij}}  \label{eqn:list 1} \\
&\bar{B}_i^a \qquad : \bar{V}_i^a  \xrightarrow{\text{projection onto summand}} W_i^{a+d_i}  \label{eqn:list 2} \\
&_i^{a-d_{ii}}\bar{\square}_i^a : \bar{V}_i^{a} \xrightarrow{\text{induced by }\Upsilon_i \text{ of \eqref{eqn:commutative diagram 1}}} \bar{V}_i^{a-d_{ii}}  \label{eqn:list 3} \\
&_i^a\bar{\square}_j^{a+d_{ij}} : V_j^{a+d_{ij}} \xrightarrow{^a_i\square_j^{a+d_{ij}}} V_i^a \xrightarrow{\eqref{eqn:induced map}} \bar{V}_i^a \label{eqn:list 4} \\
&_i^a\bar{A} \qquad : W_i^{a-d_i} \xrightarrow{_i^aA} V_i^a \xrightarrow{\eqref{eqn:induced map}} \bar{V}_i^a \label{eqn:list 5}
\end{align}
Together with $\bar{\square} = \square$, $\bar{A}=A$, $\bar{B}=B$ whenever the subscripts are different from $i$, the above maps determine a structure of framed quiver representation on
\begin{equation}
\label{eqn:bar v}
\bar{\bV} = (\bar{V}_i^a, V_j^a)_{j \neq i}^{a \in \BZ}
\end{equation}

\end{lemma}

\medskip   

\begin{proof} We must prove that the framed quiver representation $\bar{\bV}$ satisfies relations \eqref{eqn:1 bis} and \eqref{eqn:2}, \eqref{eqn:4}, \eqref{eqn:5}. Let us prove the first of these, which is the most involved one, and leave the rest as exercises to the reader. Because the composition
\[
V_{i'}^{a+d_{ii'}} \rightarrow \bar{V}_i^a \rightarrow V_{i'}^{a-d_{ii'}}	
\]
coincides with the same-named composition for $V$ instead of $\bar{V}$, then relation \eqref{eqn:1 bis} for any other $i' \neq i$ follows. It remains to check \eqref{eqn:1 bis} for the same $i$ as in Lemma \ref{lem:rep}. To this end, for any $j\neq i$ and $\ell \in \{0,\dots,-c_{ij}-1\}$, consider the composition
\begin{equation}
\label{eqn:composition}
\bar{V}_i^0 \xrightarrow{\left(^{1}_i\bar{\square}_{i}^{-\ell-2p} \right) \left(^{-\ell-2p}_i\bar{\square}_{j}^{-\ell - p} \right) 
\left(^{-\ell-p}_j\bar{\square}_{i}^{-\ell} \right) \left(^{-\ell}_i\bar{\square}_{i}^{0} \right)} \bar{V}_i^{1} \end{equation}
Above and hereafter, the superscript $\ell \in \BZ$ or $\ell \in \frac 12 + \BZ$ should be interpreted as $a+\ell d_{ii}$, and we abbreviate $c_{ij} = 2p < 0$ in order to keep our formulas concise. The first map (from right to left) in the composition \eqref{eqn:composition} takes any element
\begin{equation}
\label{eqn:initial}
\left(w_i^{\frac 12}, v_{j'}^{\ell'}\right)^{j' \in I, c_{ij'}<0}_{p' +1 \leq \ell' \leq -p'} \in \text{Ker }\Phi_i^0
\end{equation}
(we write $j'$ for any index that may differ from $j$ in \eqref{eqn:composition}, and let $c_{ij'} = 2p' < 0$) to 
\[
\left(\begin{cases} w_i^{\frac 12} &\text{if } \ell = 0 \\ 0 &\text{otherwise}\end{cases} , {(-1)^{\frac {\bar{\ell'}-\ell'}{2p'}}} {^{\bar{\ell'}}_{j'}\square_{j'}^{\ell'}} \left(v_{j'}^{\ell'}\right) \right)^{j' \in I, c_{ij'}<0}_{\text{unique } \bar{\ell'} \equiv_{2p'} \ell' \text{ s.t. } p'-\ell+1 \leq \bar{\ell'} \leq -p' - \ell} \in \text{Ker }\Phi_i^{-\ell}
\]
The second map (from right to left) in \eqref{eqn:composition} takes the above element to
\[
v_j^{-\ell-p} \in V_j^{-\ell-p}
\]
In order to apply the third map (from right to left) in the composition \eqref{eqn:composition}, let us recall that it is given by the two-step process \eqref{eqn:list 4}. As such, we must first calculate 
\[
r = \ ^{-\ell-2p}_i\square_j^{-\ell-p} \left(v_j^{-\ell-p} \right) \in V_i^{-\ell-2p}
\]
and then the third map (from right to left) in \eqref{eqn:composition} will produce the following output
\[
\left( B_i^{-\ell-2p}(r),  {^{\ell'-\ell-2p}_{j'}\square_i^{\ell'-\ell+p'-2p}} \square_i^{-\ell-2p} (r) \right)^{j' \neq i}_{p'+1 \leq \ell' \leq - p'} \in \text{Ker }\Phi_i^{-\ell-2p}
\]
Finally, the fourth map sends the above element to
\begin{equation}
\label{eqn:final 1}
\left( \begin{cases} B_i^{-\ell-2p}(r) &\text{if } \ell = -2p-1\\ 0 &\text{otherwise}\end{cases}, \right.
\end{equation}
\[
\left. \underbrace{ {(-1)^{\frac {\bar{\ell'}-\ell'+\ell+2p}{2p'}}} {^{\bar{\ell'}}_{j'}\square_{j'}^{\ell'-\ell-2p}} \square_i^{\ell'-\ell+p'-2p} \square_i^{-\ell-2p}}_{ \stackrel{\eqref{eqn:2}}= {^{\bar{\ell'}}_{j'}\square_{i}^{\bar{\ell'}+p'}} \square_i^{-\ell-2p}}  (r) \right)^{j' \neq i}_{\text{unique } \bar{\ell'} \equiv_{2p'} \ell'-\ell-2p  \text{ s.t. } p'+2 \leq \bar{\ell'} \leq -p' +1} 
\]
in $\text{Ker }\Phi_i^1$. Meanwhile, $_i^1\bar{A}\bar{B}_i^0$ takes \eqref{eqn:initial} to
\begin{equation}
\label{eqn:final 2}
\left( B_i^1A(w_i^{\frac 12} ), {^{\bar{\ell'}}_{j'}\square_i^{\bar{\ell'}+p'}} \square_i^1 A(w_i^{\frac 12}) \right)^{j' \neq i}_{p'+2 \leq \bar{\ell'} \leq -p'+1} 
\end{equation}
We must show that the sum (over all $j\neq i$ and all $0 \leq \ell \leq -2p-1$) of \eqref{eqn:final 1} plus \eqref{eqn:final 2} is equal to 0. On the first component, the thing we need to check is the formula
\begin{equation}
\label{eqn:formula 1}
B_i^1A(w_i^{\frac 12}) + \sum_{j\neq i} B^{1}_i\square_j^{p+1}(v_j^{p+1}) = 0
\end{equation}
while on the component indexed by $j' \neq i$ and $\bar{\ell'} \in \{p'+2,\dots,-p'+1\}$, we need
\begin{equation}
\label{eqn:formula 2}
 {^{\bar{\ell'}}_{j'}\square_i^{\bar{\ell'}+p'}} \square_i^1 A (w_i^{\frac 12}) + \sum_{j\neq i} \sum_{\ell = 0}^{-2p-1} {^{\bar{\ell'}}_{j'}\square_{i}^{\bar{\ell'}+p'}} \square_i^{-\ell-2p} \square_j^{-\ell-p} (v_j^{-\ell-p} ) = 0
\end{equation}
Both formulas \eqref{eqn:formula 1} and \eqref{eqn:formula 2} are immediate consequences of the fact that \eqref{eqn:initial} lies in $\text{Ker }\Phi_i^0$ (one also needs to use \eqref{eqn:5} as part of the proof of \eqref{eqn:formula 1}). \end{proof}

\medskip 

\noindent It is easy to see that the dimension of $\bar{\bV}$ of \eqref{eqn:bar v} is precisely the function $\bS_i$ of \eqref{eqn:weyl dimension vector def} applied to $\dim \bV$. Thus, the assignment of Lemma \ref{lem:rep} induces a function
\[
P_{\bv,\bw}^{\theta\text{-stab}} \rightarrow P_{\bS_i(\bv),\bw}
\]
To conclude the proof of Proposition \ref{prop:technical}, and with it that of Theorem \ref{thm:main}, we must show that the function above takes values in the $s_i(\theta)$-stable locus. We do this by adapting the argument in \cite[Section 3(iii)]{Na1}, as follows.

\medskip  

\begin{proof} \emph{of Proposition \ref{prop:technical}:} Let us write $s_i(\theta) = \bar{\theta}$. It is an immediate consequence of \eqref{eqn:stability condition} and the fact that $s_i(\omega_j) = \omega_j - \delta_{ij} \alpha_i$ that the coefficients of $\bar{\theta}$ are given by
\begin{equation}
\label{eqn:bar theta}
\bar{\theta}_j = \theta_j - \theta_i c_{ji}
\end{equation}
Consider framed quiver representations $\bV$ and $\bar{\bV}$ as in Lemma \ref{lem:rep}, and assume we have a subrepresentation
\begin{equation}
\label{eqn:subrepresentation}
\bar{\bU} \text{ consisting of } \Big\{\bar{U}_i^a \subseteq \bar{V}_i^a, \ U_j^a \subseteq V_j^a \Big\}_{j \neq i}^{a \in \BZ}
\end{equation}
which is contained inside $\text{Ker }\bar{B}$. Then we consider the collection of vector subspaces $\bU \subseteq \bV$ defined by keeping all $U_j^a$ for $j \neq i$ as above and letting $U_i^a$ denote the image of the oblique map in the following diagram
\[
\xymatrix{0 \ar[r] & \bar{U}_i^{a-d_{ii}} \ar[r] \ar@{^{(}->}[d] &  0 \bigoplus_{j \neq i} \bigoplus_{\ell = \frac {c_{ij}}2}^{-\frac {c_{ij}}2-1} U_j^{a+\ell d_{ii}} \ar[rd] \ar@{^{(}->}[d] &  & \\ 0 \ar[r] & \bar{V}_i^{a-d_{ii}} \ar[r] &  W_i^{a-d_{i}} \bigoplus_{j \neq i} \bigoplus_{\ell = \frac {c_{ij}}2}^{-\frac {c_{ij}}2-1} V_j^{a+\ell d_{ii}} \ar[r] & V_i^a \ar[r] & 0}
\]
If we let $u_i^a = \dim U_i^a$, $u_i = \sum_{a \in \BZ} u_i^a$, $\bu = (u_i^a)_{i \in I}^{a \in \BZ}$ (and similarly for $\bar{u}$), we have
\[
u_i^a \leq - \bar{u}_i^{a-d_{ii}} + \sum_{j \neq i} \sum_{\ell = \frac {c_{ij}}2}^{-\frac {c_{ij}}2-1} u_j^{a+\ell d_{ii}} \quad \Rightarrow \quad u_i \leq - \bar{u}_i - \sum_{j \neq i} c_{ij} u_j
\]
Therefore, we have
\begin{equation}
\label{eqn:ineq}
(\theta, \bu) = d_i \theta_i u_i + \sum_{j \neq i} d_j \theta_j u_j \geq  - d_i \theta_i \bar{u}_i +  \sum_{j \neq i} (- d_i c_{ij} \theta_i + d_j \theta_j) u_j \stackrel{\eqref{eqn:bar theta}}= (\bar{\theta}, \bar{\bu}) 
\end{equation}
In what follows, we will write $U_i = \oplus_{a \in \BZ} U_i^a$ etc. The fact that $(_i\square_j)(U_j) \subseteq U_i$ for any $j \neq i$ is a trivial consequence of the fact that mapping from $\bar{V}_i$ to $V_j$ is via the natural projection onto a direct summand. Meanwhile, the fact that $(_j\square_i)(U_i) \subseteq U_j$ for any $j\neq i$ and that $U_i$ lies in the kernel of $B$ are simple consequences of the fact that mapping from $U_j$ to $\bar{U}_i \subseteq \bar{V}_i$ is via the map \eqref{eqn:upsilon}. Finally, the fact that $(_i\square_i)(U_i) \subseteq U_i$ follows from relation \eqref{eqn:2}. We have thus constructed a subrepresentation $\bU \subseteq \bV$ which is contained in the kernel of $B$. Since the LHS of \eqref{eqn:ineq} is $\leq 0$ by virtue of the stability of $\bV$, we infer that the RHS is also $\leq 0$, thus proving that $\bar{\bU} \subseteq \bar{\bV}$ satisfies inequality \eqref{eqn:stable 1}.

\medskip 

\noindent Now let us consider a subrepresentation \eqref{eqn:subrepresentation} which contains $\text{Im }\bar{A}$. We consider the collection of vector subspaces $\bU \subseteq \bV$ defined by keeping all $U_j^a$ for $j \neq i$ as above and letting $U_i^a$ denote the image of the oblique map in the following diagram
\[
\xymatrix{0 \ar[r] & \bar{U}_i^{a-d_{ii}} \ar[r] \ar@{^{(}->}[d] & W_i^{a-d_i} \bigoplus_{j \neq i} \bigoplus_{\ell = \frac {c_{ij}}2}^{-\frac {c_{ij}}2-1} U_j^{a+\ell d_{ii}} \ar[rd] \ar@{^{(}->}[d] &  & \\ 0 \ar[r] & \bar{V}_i^{a-d_{ii}} \ar[r] &  W_i^{a-d_{i}} \bigoplus_{j \neq i} \bigoplus_{\ell = \frac {c_{ij}}2}^{-\frac {c_{ij}}2-1} V_j^{a+\ell d_{ii}} \ar[r] & V_i^a \ar[r] & 0}
\]
The fact that $\bU \subseteq \bV$ is a subrepresentation which contains $\text{Im } A$ is completely analogous to the preceding paragraph, and so we omit it. We have
\[
u_i^a \leq - \bar{u}_i^{a-d_{ii}} + w_i^{a-d_i} + \sum_{j \neq i} \sum_{\ell = \frac {c_{ij}}2}^{-\frac {c_{ij}}2-1} u_j^{a+\ell d_{ii}} \quad \Rightarrow \quad u_i \leq - \bar{u}_i + w_i - \sum_{j \neq i} c_{ij} u_j
\]
where $w_i = \sum_{a \in \BZ} w_i^a$ and $w_i^a= \dim W_i^a$. Therefore, as in \eqref{eqn:ineq} we conclude that
\[
(\theta, \bu) \geq d_i \theta_iw_i + (\bar{\theta}, \bar{\bu}) 
\]
However, we also have
\begin{multline*}
(\bar \theta, \bar{\bv}) = d_i \bar{\theta}_i \bar{v}_i + \sum_{j \neq i} d_j \bar{\theta}_j v_j = \\ -d_i \theta_i (w_i-v_i-\sum_{j \neq i} c_{ij} v_j) + \sum_{j \neq i} d_j (\theta_j-c_{ji}\theta_i)v_j  = -d_i \theta_i w_i + (\theta,\bv)
\end{multline*}
Comparing the above displays shows that $(\theta,\bv-\bu) \leq (\bar{\theta},\bar{\bv}-\bar{\bu})$. Since $(\theta,\bv-\bu) \geq 0$ by virtue of the stability of $\bV$, we infer that $(\bar{\theta},\bar{\bv}-\bar{\bu}) \geq 0$, thus proving that $\bar{\bU} \subseteq \bar{\bV}$ satisfies inequality \eqref{eqn:stable 2}. \end{proof}

\medskip 

\subsection{The ungraded case} Let us speculate on the ungraded analogues of the spaces in Subsections \ref{sub:quiver variety new} and \ref{sub:stability}. For any $v = (v_i)_{i \in I}$ and $w = (w_i)_{i \in I} \in \nn$, let
\[
P_{v,w} = \Big\{ V_j \xleftarrow{_j\square_i} V_i,\ V_i \xleftarrow{A_i} W_i \xleftarrow{B_i} V_i \Big \}_{i,j \in I, c_{ij} \neq 0}
\]
(with $\dim V_i = v_i$, $\dim W_i = w_i$) denote the stack of linear maps which satisfy relations \eqref{eqn:2}, \eqref{eqn:1 bis}, \eqref{eqn:4}, \eqref{eqn:5} without any of the superscripts. Let $P_{v,w}^{\theta\text{-stab}} \subset P_{v,w}$ denote the open subset of $\theta$-stable points as in Definition \ref{def:stable}. We conjecture that
\[
\mathscr{S}_i : P_{v,w}^{\theta\text{-stab}} \rightarrow P_{S_i(v),w}^{s_i(\theta)\text{-stab}}
\]
analogous to the functions $\pmb{\mathscr{S}}_i$, satisfy the Weyl group relations (see \cite{L, M, Na1} in the simply-laced case, and \cite{TY} in a related, albeit technically different setup).


\bigskip

\begin{thebibliography}{XXX}

\bibitem{C}
Chari V., 
{\em Braid group actions and tensor products},
Int. Math. Res. Not. 2002, no. 7, 357-382.

\bibitem{CM}
Chari V., Moura A.,
{\em Characters and blocks for finite-dimensional representations of quantum affine algebras},
Int. Math. Res. Not. 2005, no. 5, 257-298.

\bibitem{CP}
Chari V., Pressley A.,
{\em A guide to quantum groups},
Cambridge University Press (1995).

\bibitem{DWZ}
Derksen H., Weyman J., Zelevinsky A.,
{\em Quivers with potentials and their representations II: applications to cluster algebras}, 
J. Amer. Math. Soc. 23 (2010), no. 3, 749-790.

\bibitem{FH}
Frenkel E., Hernandez D.,
{\em Extremal monomial property of $q$-characters and polynomiality of the $X$-series}, 
J. Algebraic Combin. 62 (2025), no. 3, Paper No. 41, 40 pp.

\bibitem{FM}
Frenkel E., Mukhin E.,
{\em Combinatorics of $q$-characters of finite-dimensional representations of quantum affine algebras}, 
Commun. Math. Phys. 216, 23-57 (2001).

\bibitem{FR}
Frenkel E., Reshetikhin N.,
{\em The $q$-characters of representations of quantum affine algebras and deformations of $W$-Algebras}, Contemp. Math. 248 (1999), 163-205.

\bibitem{GLS}
Geiss C., Leclerc B., Schr\"oer J., 
{\em Quivers with relations for symmetrizable Cartan matrices I: Foundations},
Invent. Math. 209 (2017), 61-158.

\bibitem{HL}
Hernandez D., Leclerc B.,
{\em A cluster algebra approach to $q$-characters of Kirillov-Reshetikhin modules}, 
J. Eur. Math. Soc. 18 (2016), 1113-1159.

\bibitem{KKOP}
Kashiwara M., Kim M., Oh S.-J., Park E., 
{\em Monoidal categorification and quantum affine algebras II}, 
Invent. Math. 236 (2024), no. 2, 837-924.

\bibitem{K}
King A.,
{\em Moduli of representations of finite dimensional algebras},
Q. J. Math. 45, Issue 4 (1994), 515-530.

\bibitem{L}
Lusztig G.,
{\em Quiver varieties and Weyl group actions},
Ann. Inst. Fourier (Grenoble) 50 (2000), 461–489.

\bibitem{M}
Maffei A.,
{\em A remark on quiver varieties and Weyl groups},
Ann. Sc. Norm. Super. Pisa Cl. Sci., Serie 5, Volume 1 (2002) no. 3, 649-686.

\bibitem{Na0}
Nakajima H.,
{\em Quiver varieties and finite dimensional representations of quantum affine algebras},
J. Amer. Math. Soc. 14 (2001), issue 1, 145-238.

\bibitem{Na1}
Nakajima H.,
{\em Reflection functors for quiver varieties and Weyl group actions},
Math. Ann. 327, 671–721 (2003). 

\bibitem{Na2}
Nakajima H.,
{\em Quiver varieties and $t$-analogs of $q$-characters of quantum affine algebras},
Ann. of Math. 160 (2004), 1057-1097.

\bibitem{N1}
Negu\cb{t} A., 
{\em Category $\CO$ for quantum loop algebras},
Duke Math. J. 175 (2026), Issue 10.

\bibitem{N2}
Negu\cb{t} A., 
{\em Quiver moduli and quantum loop algebras},
ar$\chi$iv:2509.20815.

\bibitem{ST}
Savage A., Tingley P., 
{\em Quiver Grassmannians, quiver varieties and the preprojective algebra}, 
Pacific J. Math. 251 (2011), no. 2, 393-429.

\bibitem{TY}
Terada R., Yamakawa D.,
{\em Symmetries of Quiver Schemes}, Algebr. Represent. Theor. 28, 841-871 (2025).

\bibitem{VV}
Varagnolo M., Vasserot E., 
{\em Non-symmetric quantum loop groups and $K$-theory}, 
Ann. Sci. Éc. Norm. Supér. (4) 59 (2026), no. 1, 1-66.

\bibitem{Y}
Yamakawa D.,
{\em Quiver varieties with multiplicities, Weyl groups of non-symmetric Kac-Moody algebras, and Painlev\'e equations}, 
SIGMA 6 (2010), 087, 43 pages.

\end{thebibliography}
\end{document}